\documentclass[12pt,letterpaper]{amsart}
\usepackage{geometry}
\usepackage{amsmath,amssymb,amsfonts,amsthm,amscd}
\usepackage{mathrsfs,latexsym,url,graphicx,setspace}
\usepackage{mathpazo} %Fonts
\usepackage{color}

\usepackage[active]{srcltx}
\newtheorem{theorem}{Theorem}
\newtheorem{proposition}{Proposition}
\newtheorem{lemma}{Lemma}
\newtheorem{corollary}{Corollary}

\theoremstyle{remark}

\newtheorem{remark}{\bf Remark}

\newcommand{\C}{\mathbb{C}}

\newcommand{\D}{\Omega}

\newcommand{\Dc}{\overline{\Omega}}
\newcommand{\dbar}{\overline{\partial}}
\usepackage{hyperref}
\onehalfspace

\title[Compactness of products of Hankel operators]{Compactness of products of
Hankel operators on the polydisk and  some product domains in $\C^2$}

\author{\u{Z}eljko \u{C}u\u{c}kovi\'c}
\author{S\"{o}nmez \c{S}ahuto\u{g}lu}

\address{University of Toledo, Department of Mathematics,  Toledo, OH
43606, USA}
\email[\u{Z}eljko \u{C}u\u{c}kovi\'c]{zcuckovi@math.utoledo.edu}
\email[S\"{o}nmez \c{S}ahuto\u{g}lu]{sonmez.sahutoglu@utoledo.edu}

\subjclass[2000]{47B35}
\keywords{Hankel operators, Berezin transform}

\date{\today}
\begin{document}
\begin{abstract}
Let $\mathbb{D}^n$ be the polydisk in $\C^n$ and the symbols
$\phi,\psi\in C(\overline{\mathbb{D}^n})$  such that $\phi$ and $\psi$ are
pluriharmonic on any $(n-1)$-dimensional polydisk in the boundary of
$\mathbb{D}^{n}.$ Then $H^*_{\psi}H_{\phi}$ is compact on 
$A^2(\mathbb{D}^n)$ if and only if for every $1\leq j,k\leq n$ such that  $j\neq
k$ and any $(n-1)$-dimensional polydisk $D$, orthogonal to the $z_j$-axis in the
boundary of $\mathbb{D}^n,$ either $\phi$ or $\psi$ is holomorphic in $z_k$ on 
$D.$ Furthermore, we prove a different  sufficient condition for compactness of
the products of Hankel operators. In $\C^2,$  our techniques can be used to
get a necessary condition  on some product domains involving annuli.
\end{abstract}

\maketitle

\section*{Introduction}
In this paper we would like to understand how compactness of products of
Hankel operators interacts with the behavior of the symbols on the boundary.  We
choose to work on the polydisk and some other product domains in $\C^2.$
However, we believe that this approach could be useful on more general 
domains.

Let $\D$ be a domain in $\C^n$ and $dV$ denote the Lebesgue volume measure on
$\D.$ The Bergman space $A^2(\D)$ is the closed subspace of $L^2(\D,dV)$
consisting of all holomorphic functions on $\D.$ The Bergman projection $P$ is
the orthogonal projection from $L^2(\D)$ onto $A^2(\D).$ For a function $\phi\in
L^{\infty}(\D),$ the Toeplitz operator $T_{\phi}:A^2(\D)\to A^2(\D)$ is defined
by $T_{\phi}=PM_{\phi}$ where $M_{\phi}$ is the multiplication operator by
$\phi.$

In their famous paper, Brown and Halmos \cite{BrownHalmos63/64} introduced 
Toeplitz operators on the Hardy space on the unit disk $\mathbb{D}$ of the
complex plane and discovered the most fundamental algebraic properties of these
operators. The corresponding questions for the Bergman space remained
elusive for several decades. In 1991, Axler and the first author
\cite{AxlerCuckovic91} characterized commuting Toeplitz operators with harmonic
symbols on $\mathbb{D}$ and thus obtained an analogue of the corresponding
theorem of Brown and Halmos. In 2001, Ahern and the first author
\cite{AhernCuckovic01} studied when a product of two Toeplitz operators is equal
to another Toeplitz operator. They considered bounded harmonic functions $\phi$
and $\psi,$ and a bounded $C^2$-symbol $\xi$ with bounded invariant Laplacian.
Their main result is that $T_{\phi}T_{\psi}=T_{\xi}$ if and only if $\phi$ is
conjugate holomorphic or $\psi$ is holomorphic. Later Ahern \cite{Ahern04}
removed the  assumption on $\xi$ and assumed that $\xi\in
L^{\infty}(\mathbb{D})$ only. One of the consequences of the main result in
\cite{AhernCuckovic01} is that the semicommutator of Toeplitz operator,
$T_{\phi}T_{\psi}-T_{\phi\psi}=0,$ only in trivial cases. This result was
obtained earlier by Zheng \cite{Zheng89}, using different methods. In fact,
Zheng characterized compact semicommutators of Toeplitz operators with harmonic
symbols on the unit disk. If $\phi=\phi_1+\overline\phi_2$ and
$\psi=\psi_1+\overline\psi_2$ are bounded and harmonic on $\mathbb{D},$ where
$\phi_1,\phi_2,\psi_1,$ and $\psi_2$ are holomorphic, then compactness of
$T_{\phi}T_{\psi}-T_{\phi\psi}$ is equivalent to the condition  
\[ \lim_{|z|\to 1} \min\{ (1-|z|^2)|\phi_1'(z)|,(1-|z|^2)|\psi_2'(z)|\}=0\]
Later several authors \cite{DingTang01,ChoeKooLee04} extended this
result to the Bergman space of the polydisk $\mathbb{D}^n$ and in 2007, Choe,
Lee, Nam, and Zheng \cite{ChoeLeeNamZheng07} found characterizations of
compactness of $T_{\phi}T_{\psi}-T_{\xi}$ on the polydisk, thus extending
Ahern's result. 

A semicommutator of two Toeplitz operators can be expressed in terms of Hankel
operators. For $\phi\in L^{\infty}(\D),$ the Hankel operator
$H_{\phi}:A^2(\D)\to A^2(\D)^{\perp}$ is defined by $H_{\phi}=(I-P)M_{\phi}.$
The following relation between Toeplitz operators and Hankel operators is well
known: 
\[T_{\overline\phi \psi}-T_{\overline\phi}T_{\psi}=H_{\phi}^*H_{\psi}. \]
Thus the semicommutator can be expressed as a product of an adjoint of a Hankel
operator with another Hankel operator. Our approach is also motivated by our
previous paper \cite{CuckovicSahutoglu09} in which we studied compactness of one
Hankel operator on pseudoconvex domains in $\C^n$ in terms of the behavior of
the symbol of the operator on disks in the boundary. Thus, when faced with the 
product of two Hankel operators, we are interested in finding how compactness of
$H_{\phi}^*H_{\psi}$ interacts with the behavior of $\phi$ and $\psi$ on
the boundary of the domain.  

We finish the introduction by listing our results. Let $\xi\in
\overline{\mathbb{D}}$ and 
\[D(\xi,j)=\{(z_1,\ldots,z_{j-1},z_j,z_{j+1},\ldots,z_n)\in \mathbb{D}^{n}:
z_j=\xi \}.\]
\begin{theorem} \label{ThmPolydisk}
Let $\mathbb{D}^n$ be the polydisk in $\C^n, n\geq 2,$ and the symbols
$\phi,\psi\in C(\overline{\mathbb{D}^n})$ such that $\phi|_{D(\xi,j)}$ and
$\psi|_{D(\xi,j)}$ are pluriharmonic for all $1\leq j\leq n$ and all $|\xi|=1.$
Then $H^*_{\psi}H_{\phi}$ is compact on $A^2(\mathbb{D}^n)$ if and only if for
any $1\leq j,k\leq n$ such that $j\neq k$ and $|\xi|=1,$ either
$\phi|_{D(\xi,j)}$ or $\psi|_{D(\xi,j)}$ is holomorphic in $z_k$ on  $D(\xi,j).$
\end{theorem}
In $\C^2$ the above theorem immediately implies the following
corollary.
\begin{corollary}
\label{CorBidisk}
Let $\mathbb{D}^2$ be the bidisk in $\C^2$ and the symbols
$\phi,\psi\in C(\overline{\mathbb{D}^2})$ such that $\phi\circ g$ and $\psi\circ
g$ are harmonic for all holomorphic $g:\mathbb{D}\to \partial\mathbb{D}^2.$ Then
$H^*_{\psi}H_{\phi}$ is compact on $A^2(\mathbb{D}^2)$ if and only if for any
holomorphic function $g:\mathbb{D}\to \partial\mathbb{D}^2,$ either $\phi\circ
g$ or $\psi\circ g$ is holomorphic.
\end{corollary}

\begin{remark}
Let $\phi(z_1,z_2)=\chi_1(z_1,z_2)+z_1\overline{z}_2$ and
$\psi(z_1,z_2)=\chi_2(z_1,z_2)+\overline{z}_1z_2$ where $\chi_1,\chi_2\in
C^{\infty}_0(\mathbb{D}^2).$ Then $\phi$ and $\psi$ are smooth
functions but their restrictions on $\partial\mathbb{D}^2$ cannot be
extended onto $\overline{\mathbb{D}^2}$ as pluriharmonic functions. So unlike
the results in \cite{DingTang01,ChoeLeeNamZheng07} Theorem \ref{ThmPolydisk}
applies  to such symbols and provides many examples of (non-zero) compact
products of Hankel operators. Hence our result generalizes the previously
mentioned results in the sense that our symbols do not have to  be pluriharmonic
on $\mathbb{D}^n.$ On the other hand, we require the symbols to be continuous up
to the boundary. 
\end{remark}

In fact our method can be used to remove the plurihamonicity condition on the
symbols when proving the sufficiency, if we are willing to assume more about
the symbols.

\begin{theorem}\label{ThmSufficient}
Let $\mathbb{D}^n$ be the polydisk in $\C^n, n\geq 2,$ and the symbols 
$\phi,\psi\in C^1(\overline{\mathbb{D}^n}).$ Assume that for any holomorphic
function $g:\mathbb{D}\to \partial\mathbb{D}^n,$ either $\phi\circ g$ or
$\psi\circ g$ is holomorphic. Then $H^*_{\psi}H_{\phi}$ is compact on
$A^2(\mathbb{D}^n).$ 
\end{theorem}

We also would like to note that the sufficient condition in Theorem
\ref{ThmSufficient} is not necessary. For example, Thereom \ref{ThmPolydisk}
 implies that  $H_{\phi}^*H_{\psi}$ is compact  on $A^2(\mathbb{D}^3)$ for $\phi
(z_1,z_2,z_3)=\overline{z}_1z_2$ and $\psi (z_1,z_2,z_3)=z_1\overline{z}_2.$
However, $\phi(\xi,\xi,z_3) =\psi(\xi,\xi,z_3)=|\xi|^2$ is not holomorphic. 

Our technique can also be applied to some other product domains. 
\begin{theorem}\label{ThmAnnulus} 
Let $\D=U\times V\subset \C^2$ where $U$ and $V$ are annuli or disks in $\C,$
and the symbols $\phi,\psi\in C(\Dc).$ Assume that the restrictions of $\phi$
and $\psi$ on any disk or annulus in the boundary of $\D$ are of the form $f
+\overline g,$ where $f$ and $g$ are holomorphic and continuous up to the
boundary. If $H^*_{\psi}H_{\phi}$ is compact on $A^2(\D)$ then for any
holomorphic function $g:\mathbb{D}\to \partial\D$ either $\phi\circ g$ or
$\psi\circ g$ is holomorphic.
\end{theorem}

Commutators of Toeplitz operators are connected to products of Hankel operators
as  follows: 
\[[T_{\phi},T_{\psi}]=
H_{\overline{\psi}}^*H_{\phi}-H_{\overline{\phi}}^*H_{\psi}.\]
Hence, Theorem \ref{ThmSufficient} implies the following corollary.
\begin{corollary}
 Let $\mathbb{D}^n$ be the polydisk in $\C^n$ and the symbols $\phi,\psi\in
C^1(\overline{\mathbb{D}^n})$ be non-constant. Assume that for any holomorphic
function $g:\mathbb{D}\to \partial\mathbb{D}^n,$ either $\phi\circ g$ and
$\psi\circ g$ are holomorphic or $\overline{\phi}\circ g$ and
$\overline{\psi}\circ g$ are holomorphic. Then $[T_{\phi},T_{\psi}]$ is compact
on $A^2(\mathbb{D}^n).$ 
\end{corollary}

%%%%%%%%%%%%%%%%%%%%%%%%
\section*{Proof of Theorems \ref{ThmPolydisk} and \ref{ThmSufficient}}
%%%%%%%%%%%%%%%%%%%%%%%%%%%%

One of the important tools we need is the Berezin transform of an integrable
function $f$ on the polydisk in $\C^n$ which is defined as
$B(f)(z)= \int_{\mathbb{D}^n} f(w) |k^n_z(w)|^2 dV(w).$ Here $k^n_z(w)$  denotes
the normalized Bergman kernel of $\mathbb{D}^n$. More generally, the Berezin
transform of a bounded operator $T$ is defined as $B(T)(z)=\langle
Tk^n_z,k^n_z \rangle_{L^2(\mathbb{D}^n)}.$

\begin{proof}[Proof of Theorem \ref{ThmPolydisk}] 
 We will use the fact that if  an operator $T$ is compact then $\langle
Tf_j,f_j\rangle_{L^2(\mathbb{D}^n)}$ converges to zero whenever $\{f_j\}$
converges to zero weakly.
Let us asume that $H_{\psi}^*H_{\phi}$ is  compact and $\phi|_{D(z_0,j)}$ and
$\psi|_{D(z_0,j)}$ are pluriharmonic for all $1\leq j\leq n$ and $|z_0|=1.$
Without loss of generality let us choose $j=n$ and let us denote $z=(z',z_n)$
where $z'=(z_1,\ldots,z_{n-1})$ and define 
$\phi_{0}(z)=\phi(z)-\phi(z',z_{0}), \psi_{0}(z)=\psi(z)-\psi(z',z_{0}),$ and
denote $\psi_{z_0}(z)=\psi_{1}(z')=\psi(z',z_{0})$ and
$\phi_{z_0}(z)=\phi_{1}(z')=\phi(z',z_{0}).$ Let us fix 
$F\in A^2(\mathbb{D}^{n-1})$ with $\|F\|_{L^2(\mathbb{D}^{n-1})}\leq 1$ and
choose  a sequence $\{p_j\}\subset \mathbb{D}$ such that $p_j\to z_0.$ 
Now we define   $f_{j}(z)=F(z')k_{p_j}(z_{n})$ where $k_{p_j}$ is the
normalized Bergman kernel for $\mathbb{D} $ at $p_j.$ 
 We note that $\phi_{0}(z',z_{0})=\psi_{0}(z',z_{0})=0$
for all $z'\in \mathbb{D}^{n-1}$ and for all $\delta >0$ the sequence
$\{\|f_{j}\|_{L^2(\mathbb{D}^n\setminus\mathbb{D}^n_{z_{0},\delta})}\}$
converges to zero, where 
$\mathbb{D}^n_{z_{0},\delta}=\{z\in \mathbb{D}^n: |z_n-z_{0}|<\delta\}.$ 
Then for $\delta> 0$ one can show that  
\begin{eqnarray*}
 \|\phi_0f_{j}\|^{2}_{L^{2}(\mathbb{D}^n)}+\|\psi_{0}f_{j}\|^{2}_{L^{2}(\mathbb{
D}^n)} &\leq& \sup
\{|\phi_{0}(z)|^{2}:z\in  \mathbb{D}^n_{z_{0},\delta}
\}\|f_{j}\|^{2}_{L^{2}(\mathbb{D}^n)}\\
 && +\sup \{|\psi_{0}(z)|^{2}:z\in
\mathbb{D}^n_{z_{0},\delta}\}\|f_{j}\|^{2}_{L^{2}(\mathbb{D}^n)}
\\
 &&+  \sup \{|\phi_{0}(z)|^{2}:z\in \overline{\mathbb{D}^n}
\}\|f_{j}\|^{2}_{L^{2}(\mathbb{D}^n \setminus
\mathbb{D}^n_{z_{0},\delta})}\\
 && +\sup \{|\psi_{0}(z)|^{2}:z\in \overline{\mathbb{D}^n} \}
\|f_{j}\|^{2}_{L^{2}(\mathbb{D}^n \setminus
\mathbb{D}^n_{z_{0},\delta})}.
 \end{eqnarray*}
For any $\varepsilon>0$ we can choose $\delta>0$ so that 
\[\sup \{|\phi_{0}(z)|^{2}:z\in  \mathbb{D}^n_{z_{0},\delta} \}+\sup
\{|\psi_{0}(z)|^{2}:z\in \mathbb{D}^n_{z_{0},\delta} \}  <\varepsilon/2.\]
Furthermore, we can choose $j_{\varepsilon,\delta}$ so that   
\[\|f_{j}\|^{2}_{L^{2}(\mathbb{D}^n \setminus
\mathbb{D}^n_{z_{0},\delta})}<\frac{\varepsilon}{2\sup
\{|\phi_{0}(z)|^{2}:z\in \overline{\mathbb{D}^n} \}+ 2\sup
\{|\psi_{0}(z)|^{2}:z\in \overline{\mathbb{D}^n} \}+1} \]
for all $j\geq j_{\varepsilon,\delta}.$ Combining the above inequalities with
the fact that
$\|f_{j}\|_{L^{2}(\mathbb{D}^n)}\leq 1 $ we get
$\|\phi_{0}f_{j}\|^{2}_{L^{2}(\mathbb{D}^n)}+\|\psi_{0}f_{j}\|^{2}_{L^{2}
(\mathbb{D}^n)} <
\varepsilon$ for  $j
\geq j_{\varepsilon,\delta}.$ This implies that  
\[\|H_{\phi_{0}}(f_{j})\|_{L^{2}(\mathbb{D}^n)}+\|H_{\psi_{0}}(f_{j})\|_{L^{2}
(\mathbb{D}^n)}\to
0 \text{ as } j\to \infty.\] 
The above statement together with the assumption that
$H^*_{\psi}H_{\phi}$ is compact and $H_{\phi}=H_{\phi_{z_0}}+H_{\phi_{0}}$
and  $H_{\psi}=H_{\psi_{z_0}}+H_{\psi_{0}}$ imply that   
$\langle H_{\phi_{z_0}}(f_{j}),H_{\psi_{z_0}}(f_{j})
\rangle_{L^{2}(\mathbb{D}^n)}$
converges to zero.  Using the fact that $\mathbb{D}^n$ is the polydisk and the
function $\phi_{z_0}$ depends only on $z'$ one can show
that $H_{\phi_{z_0}}(f_{j})(z)=H_{\phi_{1}}(F)(z')k_{p_{j}}(z_n)$ and  
\[ \langle H_{\phi_{z_0}}(f_{j}),
H_{\psi_{z_0}}(f_{j})\rangle_{L^{2}(\mathbb{D}^n)} =
\langle H_{\phi_{1}}(F),H_{\psi_{1}}( F)\rangle_{L^{2}(\mathbb{D}^{n-1})} \|
k_{p_j}\|^{2}_{L^{2}(\mathbb{D})}. \]
Then compactness of  $H^*_{\psi}H_{\phi}$ implies that
 $\langle H_{\phi_{1}}(F), H_{\psi_{1}} (F)\rangle_{L^{2}(\mathbb{D}^{n-1})} =
0$ for all $F\in A^2(\mathbb{D}^{n-1}).$  Now Theorem 2.3 in
\cite{DingTang01} implies that for any $1\leq k\leq n-1$ either $\phi_1$ or
$\psi_1$ is holomorphic in $z_k.$ Therefore, for any $1\leq k\leq n-1$ either
$\phi$ or $\psi$ is holomorphic in $z_k.$

To prove the other direction of the theorem, let $q$ be a boundary point of
$\mathbb{D}^n$ and $k^{n}_{q_j}$ denote the normalized Bergman kernel of
$\mathbb{D}^n$ centered at $q_j\in \mathbb{D}^n$ where $q_j \to q.$ First, we
will show that   
$\langle H_{\phi}k^n_{q_j},H_{\psi}k^n_{q_j}\rangle_{L^2(\mathbb{D}^n)}$ 
converges to zero. Then we will use the fact (\cite{AxlerZheng98,Englis99}, see
also \cite[Theorem 2.1]{ChoeKooLee09}) that $H_{\psi}^*H_{\phi}$ is compact if
and only if $B(H_{\psi}^*H_{\phi}) \in C_0(\mathbb{D}^n)$ where  $C_0(\Omega)$
denotes the class of functions that are continuous on $\Omega$ and have zero
boundary limits. It is easy to see that $B(H_{\psi}^*H_{\phi}) \in
C(\mathbb{D}^n).$ There exists $1\leq j\leq n$ and $\xi\in \C$ such that
$|\xi|=1$ and $q\in D(\xi,j).$ We extend $\psi|_{D(\xi,j)}$ and
$\phi|_{D(\xi,j)}$ trivially in $z_j$ so that the extensions, $\psi_1$ and
$\phi_1$, are independent of $z_j$ variable and are continuous up to the
boundary of $\mathbb{D}^n.$ Let us define
$\phi_0=\phi-\phi_1$ and $\psi_0=\psi-\psi_1.$ Then $\phi_0=\psi_0=0$ on
$D(\xi,j)$  and, as is done in the first part of this proof, one can show that 
both sequences $\{H_{\phi_0}k^n_{q_j}\}$ and $\{H_{\psi_0}k^n_{q_j}\}$ converge to
zero. Since $\phi_1$ and $\psi_1$ are pluriharmonic on $\mathbb{D}^n,$
continuous up to the boundary, and for each variable either $\phi_1$ or
$\psi_1$ is holomorphic, Theorem 2.3 in \cite{DingTang01} implies that  
$H_{\psi_1}^*H_{\phi_1}=0.$ Therefore, $H_{\psi}^*H_{\phi}$ is compact.
\end{proof}

In order to prove  Theorem \ref{ThmSufficient} we need the following
Lemma
  \begin{lemma}\label{Lem1}
Let $U$ be a  domain in $\C^n$ and the functions 
$\phi,\psi\in C^1(U)$ are such that for any holomorphic function 
$g:\mathbb{D}\to U$ either $\phi\circ g$ or $\psi\circ g$  is
holomorphic. Then either $\phi$ or $\psi$ is holomorphic on $U.$
\end{lemma}
\begin{proof} Let  $p,q\in U$ such that $\dbar\phi(p)\neq 0$ and 
$\dbar \psi(q)\neq 0.$ Assume that $p\neq q.$ Let $\varepsilon>0$ and 
$\gamma:[0,1]\to U$ be a curve so that $\gamma(0)=p,\gamma(1)=q,$ and 
\[\{z\in \C^n:\text{dist}(z,\gamma)<\varepsilon\}\subset U\] 
where $\text{dist} $ denotes the Euclidean distance. Using Stone-Weierstrass
Theorem we choose a complex-valued (real) polynomial $P:\mathbb{R}\to\C^n$ so
that $|P(x)-\gamma(x)|<\varepsilon/4$ for all $x\in [0,1].$ Let us define 
\[f(x)= P(x)+ x(q-P(1))+(1-x)(p-P(0)).\]
The function $f$ has a holomorphic  extension to $\C$ and we will denote the
extension by $f$ as well. Hence,  $f:\C\to \C^n$ is holomorphic such that
$f(0)=p,f(1)=q,$ and $f(z)\subset U$ for  
$z\in L=\{z\in\mathbb{R}:0\leq z\leq 1\}.$

Let $\{e_1,e_2,\ldots, e_n\}$ denote the standard basis in
$\C^n,$ and define $E_j=e_j$ for $1\leq j\leq n$ and
$E_{n+j}=\sum_{k=1}^nk^{j-1}e_k$ for $1\leq j\leq n-1.$ 
 Using Vandermonde matrix one can show that the set 
$\{E_{j_1},E_{j_2},\ldots,E_{j_n}\}$ is linearly independent for any $1\leq
j_1<j_2<\cdots <j_n\leq 2n-1.$  

Let $M>0$ and define
\[g_{j,M}(z)=f(z)+\frac{z(z-1)}{M}E_j.\] 
Let us fix $M>0$ large enough so that  $g_{j,M}(z)\in U$ for $z\in L.$  Then
there exists a simply connected neighborhood $V$ of $L$ such that 
$g_{j,M}(z)\in U$ for $z\in V.$
We choose a conformal mapping $h:\mathbb{D}\to V$ and define 
$g_j= g_{j,M}\circ h.$ Then $g_j:\mathbb{D}\to U$ for $1\leq j\leq
2n-1,$  and the sets  $\{g_{j_1}'(z_0),g_{j_2}'(z_0),\ldots,g_{j_n}'(z_0)\}$
and  $\{g_{j_1}'(z_1),g_{j_2}'(z_1),\ldots,g_{j_n}'(z_1)\}$ are  linearly
independent for $h(z_0)=0,h(z_1)=1$ and  any $1\leq j_1<j_2<\cdots<j_n\leq
2n-1.$ Since for any $j,$ either $\phi\circ g_j$ or $\psi\circ g_j$ is
holomorphic, there exist $1\leq j_1<j_2 \cdots <j_n\leq 2n-1$ such that either
$\phi\circ g_{j_k}$ is holomorphic for $1\leq k\leq n$ or $\psi\circ g_{j_k}$ is
holomorphic for $1\leq k\leq n.$ Furthermore, using the chain rule
together with linear independence of $\{g_{j_k}'(z_0):1\leq k\leq n\}$ and
$\{g_{j_k}'(z_1):1\leq k\leq n\}$ one can show that either 
$\dbar \phi (p)=\dbar\phi(q)=0$ or $\dbar \psi (p)=\dbar\psi(q)=0.$ 

If $p=q$ then one can use affine disks along $E_j$'s to show that  either
$\dbar\phi(p)=0$ or $\dbar\psi(p)=0.$   Hence, we reached a contradiction
completing the proof. 
\end{proof}

\begin{proof}[Proof of Theorem \ref{ThmSufficient}]
We will use Lemma \ref{Lem1} together with the  ideas in the  second part
of the proof  of Theorem \ref{ThmPolydisk}. For any $|\xi|=1$ and 
$1\leq j\leq n$ we decompose the symbols as $\phi=\phi_0+\phi_1$ and
$\psi=\psi_0+\psi_1$ such that 
\begin{itemize}
 \item[i.] $\phi_0=\psi_0=0$ on $D(\xi,j),$
\item[ii.] $\phi_1 |_{D(\xi,j)}=\phi|_{D(\xi,j)},\psi_1 |_{D(\xi,j)} =
\psi|_{D(\xi,j)},$  
\item[iii.] $\phi_1$ and $\psi_1$ are continuous on $\overline{\mathbb{D}^n},$
\item[iv.]  either $\phi_1$ or $\psi_1$ is holomorphic on $\mathbb{D}^n.$ 
\end{itemize}
Then either  $H_{\phi_1}=0$ or $H_{\psi_1}=0$ and both
sequences $\{H_{\phi_0}k_{q_j}\}$ and $\{H_{\psi_0}k_{q_j}\}$ converge to 0
in $L^2(\mathbb{D}^n)$ for  $q_j\to q\in D(\xi,j).$  Hence, 
$B(H_{\psi}^*H_{\phi}) \in C_0(\mathbb{D}^n)$ and in turn this implies that
$H_{\psi}^*H_{\phi}$
is compact. 
\end{proof}

%%%%%%%%%%%%%%%%%%%%%%%%
\section*{Proof of Theorem \ref{ThmAnnulus}}
%%%%%%%%%%%%%%%%%%%%%%%%
%%%%%%%%%%%%%%% PROPOSITION
To prove Theorem \ref{ThmAnnulus} one reduces the problem onto
$U$ or $V$ as in the first part of the  proof of Theorem \ref{ThmPolydisk}.
Then if the problem is reduced onto an annulus one uses the following
Proposition instead of Ding and Tang's Theorem.

\begin{proposition} \label{Prop1}
 Let $\mathcal{A}=\{z\in \C:0<r<|z|<R\}$ and $\phi$ and $\psi$ be holomorphic on
$\mathcal{A}$ and continuous on $\Dc.$ Assume that
$B(\overline\psi\phi)=\overline\psi\phi.$ Then  either $\phi$ or $\psi$ is
constant.
 \end{proposition}

\begin{proof} Let us assume that $B(\overline\psi\phi)=\overline\psi\phi.$ Then
by a result of \u{C}u\u{c}kovi\'c \cite[Theorem 9]{Cuckovic96}
$B(\overline\psi\phi)=\overline\psi\phi$ implies that 
\begin{align}\label{CuckovicResult}
\overline\psi\phi=R(\overline\psi\phi)+h
\end{align} 
where $h$ is a harmonic function and the radialization operator $R$ is defined
as $R(k)(z)=(2\pi)^{-1}\int_0^{2\pi} k(ze^{i\theta})d\theta.$ If we apply the
Laplacian to \eqref{CuckovicResult} we get $\overline\psi' \phi'=R(\Delta 
(\overline\psi\phi)). $ Hence, $\overline\psi'\phi'$ is radial. Let
$\phi'(z)=\sum_{n=-\infty}^{\infty}a_nz^n$
and $\psi'(z)=\sum_{m=-\infty}^{\infty}b_mz^m.$ Then, on one hand 
\[\overline\psi'(z)\phi'(z)=\sum_{n,m=-\infty}^{\infty}a_n\overline
b_mr^{n+m}e^{i(n-m)\xi} \]
where $z=re^{i\xi}.$  On the other hand, since $\overline\psi'\phi'$  is a
radial function we get 
\begin{align*}
 \overline\psi'(z)\phi'(z)&=
\frac{1}{2\pi}\int_0^{2\pi}\overline\psi'(ze^{i\theta})\phi'(ze^{i\theta}
)d\theta\\
&=\frac{1}{2\pi} \sum_{n,m=-\infty}^{\infty}a_n\overline b_mr^{n+m}
e^{i(n-m)\xi}\int_0^{2\pi}e^{i(n-m)\theta}d\theta\\
&= \sum_{n=-\infty}^{\infty}\frac{a_n\overline b_nr^{2n}}{2\pi}.
\end{align*}
Hence, $\sum_{n\neq m}a_n\overline b_mr^{n+m}e^{i(n-m)\xi}=0$ for all $\xi.$ We
can rewrite the last equation as 
\[\sum _{k\neq 0} \left( \sum_{m=-\infty}^{\infty}  a_{m+k}\overline
b_mr^{2m+k}\right) e^{ik\xi}=0.\]
This is a Fourier series that is equal to zero. Hence 
$\sum_{m=-\infty}^{\infty}  a_{m+k}\overline b_mr^{2m+k}=0$ for all $k\neq 0.$ 
This is a Laurent series that is equal to zero. Therefore, 
$a_{m+k}\overline b_m=0$ for all $k\neq 0$ and all $m.$ In return this implies
that if $b_{m_0}\neq 0$ then $a_m=0$ for $m\neq m_0.$ That is, either there
exists an integer $m$ and two nonzero constants $a$ and $b$ such that
$\phi(z)=az^m$ and $\psi(z)=bz^m$ or either $\phi$ and $\psi$ is
constant. Next we will show that in the first case $m=0.$
Recall that the Bergman kernel for the annulus $\{z\in \C: \rho<|z|<1\}$ is 
\[K_w(z)=\frac{1}{\pi}\sum_{-\infty}^{\infty}
\frac{n+1}{1-\rho^{2n+2}}(\overline{w}z)^n-\frac{1}{2\pi\ln\rho}(\overline{w}
z)^{-1}. \]
Without loss of generality we may assume
that $R=1$ and $r=\rho <1$ for a
fixed $m$ we have $B(|z|^{2m})(w)=|w|^{2m}.$ Then 
\[|w|^{2m}
\|K_w\|^2=\int_{\rho}^1\int_0^{2\pi}r^{2m+1}|K_w(re^{i\theta})|^2d\theta
dr.\]
The last equation can be expanded as 
\begin{multline*} 
-\frac{|w|^{ 2m-2 } } {2\pi\ln \rho}+\sum_{n\neq
-1}\frac{(n+1)|w|^{2(m+n)}}{1-\rho^{2n+2}} = \frac{\pi (1-\rho^{2m})}{|w|^2
m (2\pi  \ln \rho )^2 |w|^2} \\ 
+ \sum_{n\neq -1}\frac{(n+1)^2|w|^{2n}}{(1-\rho^{2n+2})^2}
\frac{(1-\rho^{2(n+m+1)})}{2(n+m+1)} 
\end{multline*}
Now let $k=m+n$ then $n=k-m$ and the last equation becomes
\begin{multline*} 
-\frac{|w|^{ 2m-2 } } {2\pi\ln \rho}+\sum_{k\neq m
-1}\frac{(k-m+1)|w|^{2k}}{1-\rho^{2(k-m+1)}} = \frac{\pi
(1-\rho^{2m})}{|w|^2m (2\pi\ln \rho )^2 |w|^2} \\
+ \sum_{n\neq
-1}\frac{(n+1)^2|w|^{2n}}{(1-\rho^{2n+2})^2}
\frac{(1-\rho^{2(n+m+1)})}{2(n+m+1)} 
\end{multline*}
By equating the coefficients of each term we get
\[\frac{n-m+1}{1-\rho^{2(n-m+1)}}=\frac{(n+1)^2(1-\rho^{2(n+m+1)})}{
(n+m+1)(1-\rho^{2(n+1)})^2} \]
for $n\neq -1$ and $n\neq m-1.$  Let $l=n+1$ and $\xi=\rho^2.$ Then the last
equation turns into 
\begin{equation}\label{Eq1} 
\frac{l^2-m^2}{l^2}=\frac{(1-\xi^{l+m})(1-\xi^{l-m})}{(1-\xi^l)^2}
\end{equation}
for $l\neq 0$ and $l\neq m.$ From now on we will choose $l>m.$ Let us  define
the following function 
\[ f_l(x)=\frac{(1-\xi^{l+x})(1-\xi^{l-x})}{l^2-x^2}.\]
One can show that $f_l$ is an even, nonnegative function defined on $(-l,l)$
and \eqref{Eq1} implies that $f_l(0)=f_l(m).$ Then using the logarithmic
differentiation we get
\begin{align} \label{Eq2}
(l^2-x^2)\frac{f'_l(x)}{f_l(x)}&=(l^2-x^2)\xi^l \ln \xi \left(
\frac{\xi^{-x}}{1-\xi^{l-x}}- \frac{\xi^{x}}{1-\xi^{l+x}} \right)+2x\\
 \nonumber  &=(l^2-x^2)\xi^l \ln \xi
\left(\frac{\xi^{-x}-\xi^{x}}{(\xi^x-\xi^l)(\xi^{-x}-\xi^{l})} \right)+2x
\end{align}
Power series expansions for  $\xi^x$ and $\xi^{-x}$ imply that 
\[ \xi^{-x}-\xi^x=-2\ln \xi \sum_{j=0}^{\infty} \frac{(\ln
\xi)^{2j}}{(2j+1)!}x^{2j+1} .\]
Then  there exists $0<\delta<l$ so that $ |\xi^{-x}-\xi^x|\leq -3\ln \xi | x| $
for $|x|\leq \delta.$  Now we use estimate $(\xi^{\delta}-\xi^l)^2\leq
(\xi^x-\xi^l)(\xi^{-x}-\xi^l) $ for $0<x\leq \delta$ to get 
\[(l^2-x^2)\frac{f'_l(x)}{f_l(x)} \geq 
\left(\frac{-6(\ln\xi)^2(l^2-\delta^2)\xi^l}{(\xi^\delta-\xi^l)^2}+2\right)
x \text{ for } 0\leq x \leq \delta. \]
Then since $0<\xi<1$ there exists $l_0>4m$ so that  
$(l^2-x^2)\frac{f'_l(x)}{f_l(x)} \geq \frac{x}{2}$ for $l\geq l_0$ and $0\leq
x\leq \delta.$ Hence $f_l$ are increasing functions on $[0,\delta]$ for all
$l\geq l_0.$ On the other hand for $\delta\leq x \leq m $ there exists
$l_1\geq 4m$ such that  $l\geq l_1$ implies that 
\[\left| (l^2-x^2)\xi^l \ln \xi \left( \frac{\xi^{-x}}{1-\xi^{l-x}} -
\frac{\xi^{x}}{1-\xi^{l+x}} \right) \right|\leq \delta.\]
Then \eqref{Eq2} implies that $f'_l>0$ on $[\delta, m]$ for $l\geq l_1$.
Therefore, $f_l$ are increasing functions on $[0,m]$  and $f_l(m)>0$ for $l\geq
\max\{l_0,l_1\}>4m.$ 
\end{proof}

\singlespace

\end{document}